\documentclass{article}
\usepackage[utf8]{inputenc}
\usepackage{chicago}
\usepackage{amsthm}
\usepackage{ragged2e}
\usepackage{amsmath}
\usepackage{amsthm}
\usepackage{amssymb}
\usepackage{graphicx}
\usepackage{ragged2e} 

\usepackage{xcolor, colortbl}
\usepackage{array, multirow, multicol}

\newtheorem{theorem}{Theorem}%

\newtheorem{corollary}{Corollary}%

\newtheorem{proposition}{Proposition}%

\title{Probabilistic Methods in the Study of Topological Indices on Random Spider Trees}
\author{Saylé C. Sigarreta, Saylí M. Sigarreta and Hugo Cruz-Suárez.}
\date{April 2022}

\begin{document}

\maketitle
\textbf{Abstract:} In this paper, we  characterize the structure and topological indices of  a class of random spider trees (RSTs)  such as degree-based Gini index, degree-based Hoover index, generalized Zagreb index and other indices associated with these.  We obtain the exact and asymptotic distributions of the number of leaves via probabilistic methods. Moreover, we relate this model to the class of RSTs that evolves in a preferential attachment manner.

\textbf{keywords:}Random trees, Spider trees, Gini index, Hoover index, Zagreb index, Topological indices.

\section{Introduction}\label{sec1}

Initiated in 1736 by Euler and developed in the 19th century by the englishmen A. Cayley and J.J. Silvester, Graph Theory has become a very powerful practical and theoretical tool. A graph $G$ is determined by two sets $(V,E)$, the set of nodes and edges. The edges and nodes are interpreted according to the problem to be modeled. Highlighting the trees as a very important and studied family of graphs, which from its origin has proven to have many applications in different areas. In mathematical chemistry, trees are used to characterize the molecular structure of chemical compounds; in this context the nodes represent the molecules and the edges the chemical bonds. Particularly, recent work on the spider trees includes the packing, labeling and extremal graphs problems. The packing problem has been well-studied in the literature within a wide range of variants. A special attention has been devoted to packing of trees into graphs. In this context, García et al. \cite{r1} conjectured that any two trees with $n$ vertices different from a star admit a tight
planar packing. In \cite{r1} the authors proved their conjecture for some restricted cases. Oda and Ota \cite{r2} proved it when one tree is a caterpillar or it is a spider of diameter four. Frati et al. \cite{r3} extended the last result to any spider. Finally, the conjecture was proved in \cite{r4}.  On the other hand, a graph labeling is an assignment of integers to the edges, nodes, or both, of a graph so that it meets to certain conditions. Based on the classification, there are edges labeling, node labeling, and total labeling. Graceful labeling of a tree with $n$ nodes is a labeling of its vertices with the numbers from $0$ to $n-1$, so that no two vertices share a label, labels of edges, being absolute difference of the labels of its end points, are also distinct. There is a long-standing conjecture named Graceful tree conjecture or Ringel-Kotzig-Rosa conjecture that says “All trees are graceful”. In particular Bahls, Lake, and Wertheim \cite{E1}
proved that spiders for which the
lengths of every path from the center to a leaf differ by at most one are graceful and Jampachon, Nakprasit, and Poomsa-ard \cite{S1} provide graceful labelings for some classes
of spiders.  Last but not least important, one relevant class of trees for chemical studies are the trees with a given number of pendants. A node is called pendant if it has degree 1. In \cite{r7} the authors proved that the trees with $n$ pendants ($n\geq 3$) that maximize the modified first Zagreb connection index must be spider trees or double stars. More recently, spider trees were adopted to model molecular structures of chemical compounds in mathematical chemistry. For instance, in \cite{0}, spider trees are used
to study hexagonal systems that model benzenoid molecules and unbranched catacondensed benzenoid molecules.\\

In the development of applications, it has become natural to conclude that random graphs are an appropriate and useful tool to analyze phenomena that evolve over time, since many important characteristics are difficult to capture using deterministic models. In this sense, it is important to mention that some  works  perform  studies of  topological  indices  on  random  networks and structures where they discuss, for instance, the application of specific topological indices as complexity measures for random networks. For a better treatment we refer interested readers to \cite{23,24,15,36,37,38,PP1}. In particular, motivated by the substantial increase of interests in random tree models Ren, Zhang and Dey investigated two classes of random lobster trees  which evolve according to different rules
and a class of random spider trees which grow in a preferential attachment manner \cite{2}. In this important work, the authors obtained very useful results, as a matter of fact, for the class of random spider trees they characterized
the structure of the model by determining the exact and asymptotic distributions of the number of leaves, and by computing two kinds of topological indices: Zagreb index and Gini index. Inspired by \cite{2} and considering the arguments put forward in the previous paragraphs, in this manuscript, we consider a class of spider trees that are incorporated with randomness, called random
spiders trees (RSTs), and we investigate several useful topological indices of this random class, including degree-based Gini index, degree-based Hoover index, generalized Zagreb index and other indices associated with these. Specifically, a central limit theorem is developed for the asymptotic distribution of the number of leaves of a RST.

\section{ Random Spider Trees}\label{sec2}

A spider tree is a connected tree with a centroid of degree at least 3. All the remaining nodes
are classified into two categories: internal nodes of degree 2 and leaves of degree 1, respectively. Thus, except for the centroid, all the nodes in a spider tree have degrees at most 2. The class of RSTs considered in this paper evolves in the following way: at time 1, a RST starts with a seed graph containing
a centroid  and three leaves. At each subsequent stage the leaves and centroid will be able to recruit new nodes (at time  $n$):\\
\begin{enumerate}
\item The centroid will be selected with probability $p$, $0 < p < 1$.

\item A leaf will be selected with probability $\displaystyle\frac{1-p}{L_{n}}$ where $L_{n}$ denotes the number of leaves in the  RST at time $n$, where  
\end{enumerate}
\medskip

\begin{center}
    $p+\displaystyle\sum_{i=1}^{L_{n}}{\frac{1-p}{L_{n}}}=p+\displaystyle\frac{(1-p)L_{n}}{L_{n}}=1.$
\end{center}
\medskip

\noindent
Note that only the centroid
and leaves are qualified for recruiting new nodes. If the centroid is selected, a
new leaf is attached to it; if a leaf is selected, a new leaf is attached to the (selected) leaf, and
the recruiter is converted to an internal node. Finally, we have that at each stage the generated graph is a spider tree with $n+3$ nodes.

\subsection{Leaves}\label{subsec2}

In the following, $L_{n}$ denotes the number of leaves in a RST at time $n$, with $n \geq 1$. By the construction of the model it follows that
\medskip

\begin{center}
    $L_{n}=3+\displaystyle\sum_{i=1}^{n-1}Ber(p)=3+Bin(n-1,p)$.
\end{center}
\medskip

\noindent
Consequently, the following statements are easily verified.

\begin{proposition}\label{p1} 
For $n \geq 1$ and $0 < p < 1$, the following statements hold:
\begin{enumerate}

\item $\mathbb{E}(L_{n})=3+(n-1)p$ and $\mathbb{V}(L_{n})=(n-1)(1-p)p$.

\item $M_{L_{n}}(t)=(1-p+pe^{t})^{n-1}e^{3t}$, $t \in \mathbb{R}$.
\item  For each $k \in \mathbb{R}$,~ $\displaystyle\frac{L_{n}-3-(n-1)p}{\sqrt{p(1-p)(n+k)}} \xrightarrow{D} N(0,1)$ as $n \rightarrow \infty$.
\end{enumerate}
\end{proposition}

\subsection{A class of RSTs that evolves in a preferential attachment manner }\label{subsubsec2}
In a very recent article \cite{2}, the authors inspired by the seminal paper \cite{18} introduced a class of RSTs that evolves in a preferential attachment manner as follows. At time 1, a RST starts with a seed graph containing a centroid of degree 3 and three leaves. At each subsequent point, the probability of a qualified node recruiting a newcomer is proportional to its degree. If the centroid is selected, a new leaf is attached to it; if a leaf is selected, a new leaf is attached to the (selected) leaf, and the recruiter is converted to an internal node. Consequently, for $n\geq 2$ it is given by
$$
\mathbb{P}\left(\mathbf{1}_{v,n}\right)=\frac{deg_{v,n-1}}{\displaystyle\sum_{u \in Q_{n-1}}deg_{u,n-1}},
$$
where $v$ is a qualified node at time $n$, $\mathbf{1}_{v,n}$ indicates the event that node $v$ is chosen as recruiter at time $n$, $deg_{i,n-1}$ is the degree of a node $i$ at time $n-1$ and $Q_{n-1}$ denotes the set of qualified nodes at time $n-1$. Then, for $n \geq 2$, it follows that: \\ 

\begin{enumerate}
    \item The probability that the centroid recruits a newcomer at time $n$ is $\frac{L_{n-1}}{2L_{n-1}}=\frac{1}{2}$.\\
    
    \item The probability that a leave recruits a newcomer at time $n$ is $\frac{1}{2L_{n-1}}$.
\end{enumerate}
\medskip

\noindent
 Therefore, we can conclude that the class of RSTs that evolves in a preferential attachment manner (preferential model) is the model presented in Section \ref{sec2} with $p=\displaystyle\frac{1}{2}$.
\section{Topological Indices}
The purpose of topological indices is to study the structural properties associated with a graph and its invariants using a certain numerical value. The idea of capturing the information in numerical form is to be able to compare the graphs according to the property to be studied. Let $G=(V,E)$, many important topological indices ($TI(G)$) can be defined as
\begin{equation}\label{sa}
   TI(G)= \displaystyle\sum_{v\in V} h(deg_{v})^{\alpha},
\end{equation}
    
\noindent
where $\alpha \in \mathbb{R} $, $h:\{1,2, \dots \} \rightarrow (0, \infty)$ and $deg_{v}$ is the degree of a node $v$. In Section \ref{pa}, we will study the indices that satisfy (\ref{sa}) in the model introduced in Section \ref{sec2}. At each stage the generated tree has three types of nodes, centroid, leave and internal, for which their degrees are $L_{n}$, 1 and 2, respectively.

\begin{proposition}\label{p5}
Let $TI_{n}$ the value of the topological index at stage $n$.  For each $n \geq 1$, we have\\
\begin{center}
    $\mathbb{E}(TI_{n})=\mathbb{E}(h(L_{n})^{\alpha})+(h(1)^{\alpha}-h(2)^{\alpha})\mathbb{E}(L_{n})+h(2)^{\alpha}(n+2)$ and
\end{center}
\hfill

\medskip
\begin{center}
 $\mathbb{V}(TI_{n})=\mathbb{V}(h(L_{n})^{\alpha}+(h(1)^{\alpha}-h(2)^{\alpha})L_{n})$.   
\end{center}

\end{proposition}
\begin{proof}
Note that $I_{n}+L_{n}+1=n+3$, where $I_{n}$ is the number of internal nodes in the tree at stage $n$, it follows that:\\

\begin{center}
    $TI_{n}= h(L_{n})^{\alpha}+h(1)^{\alpha}L_{n}+h(2)^{\alpha}(n+2-L_{n})$
\end{center}

\begin{equation}\label{e2}
       ~~~~~~~~~~~=h(L_{n})^{\alpha}+(h(1)^{\alpha}-h(2)^{\alpha})L_{n}+h(2)^{\alpha}(n+2).
\end{equation}
\medskip

\noindent
By (\ref{e2}), we immediately get the mean and the variance of $TI_{n}$.
\end{proof}

\medskip
\noindent
As a consequence of Proposition \ref{p1} b the next result follows.
\begin{proposition}\label{p6}
If $h(L_{n})=aL_{n}+b$ with $a,b \in \mathbb{R}$ then $M_{h(L_{n})}(t)= (1-p+pe^{at})^{n-1}e^{(3a+b)t}$, $t \in \mathbb{R}$ and $n \geq 1$.

\end{proposition}

\subsection{Generalized Zagreb Index}\label{pa}

\noindent
At time $n \geq 1$, taking $h(x)=x$ and $\alpha \in \mathbb{R}-\{0\}$ in (\ref{sa}) we obtain the generalized Zagreb index ($Z^{g}_{n}$). According to (\ref{e2}),

\hfill

\begin{equation}\label{j}
     Z^{g}_{n}= L_{n}^{\alpha}+(1-2^{\alpha})L_{n}+2^{\alpha}(n+2).
\end{equation}

\medskip
\noindent

\begin{proposition}\label{t1}
Let $\alpha \in \{1,2,\dots\}$ and $t \in \mathbb{R}$, we have

\hfill
\begin{center}
    $\displaystyle\frac{d^{\alpha} M_{L_{n}}}{dt}(t)=\displaystyle\sum_{i=1}^{\alpha}C_{\alpha,i}\hspace{0.05cm}p^{i}e^{it}\displaystyle\frac{d^{i} M_{L_{n}}}{du}(u(t))$,
\end{center}
\medskip
\noindent
 with $C_{\alpha,\alpha}=C_{\alpha,1}=1$ and $C_{\alpha,i}=C_{\alpha-1,i-1}+iC_{\alpha-1,i}$ for $i\in \{2,3, \dots ,\alpha -1\}$. 
\end{proposition}
\begin{proof}  
 We will get the proof via mathematical induction on $\alpha$. Firstly observe that, Proposition \ref{p1} b may be simplified by defining a new variable  $u(t)=1-p+pe^{t}$. Substituting $1-p+pe^{t}$ by $u(t)$ we get $M_{L_{n}}(t)=M_{L_{n}}(u(t))=u(t)^{n-1}(\frac{u(t)+p-1}{p})^{3}$. Thus, for the base, that is $\alpha =3$, we have
\begin{center}
  $\displaystyle\frac{d^{3} M_{L_{n}}}{dt}(t)=p^{3}e^{3t}\displaystyle\frac{d^{3} M_{L_{n}}}{du}(u(t))+3p^{2}e^{2t}\displaystyle\frac{d^{2} M_{L_{n}}}{du}(u(t))+pe^{t}\displaystyle\frac{d M_{L_{n}}}{du}(u(t))$,

\end{center}
with $C_{3,1}=C_{3,3}=1$ and $C_{3,2}=C_{2,1}+2C_{2,2}=3$. We assume that the statement holds for all $\alpha$, i.e.

\begin{equation}\label{s}
     \displaystyle\frac{d^{\alpha} M_{L_{n}}}{dt}(t)=\displaystyle\sum_{i=1}^{\alpha}C_{\alpha,i}\hspace{0.05cm}
p^{i}e^{it}\displaystyle\frac{d^{i} M_{L_{n}}}{du}(u(t)),
\end{equation}

\noindent
 with $C_{\alpha,\alpha}=C_{\alpha,1}=1$ and $C_{\alpha,i}=C_{\alpha-1,i-1}+iC_{\alpha-1,i}$ for $i\in \{2,3, \dots ,\alpha -1\}$. Note that for each $i\in \{1,2, \dots ,\alpha\}$ we arrive at

\medskip
\noindent
$\displaystyle\frac{d }{dt}\left( C_{\alpha,i}p^{i}e^{it}\displaystyle\frac{d^{i} M_{L_{n}}}{du}(u(t))\right)$

\begin{equation}\label{e}
    =i\hspace{0.05cm}C_{\alpha,i}\hspace{0.05cm}p^{i}e^{it}\displaystyle\frac{d^{i} M_{L_{n}}}{du}(u(t))+C_{\alpha,i}\hspace{0.05cm}p^{i+1}e^{(i+1)t}\displaystyle\frac{d^{i+1} M_{L_{n}}}{du}(u(t)).
\end{equation}

\noindent
By (\ref{s}) and (\ref{e}), we have proved the following result
\begin{equation}\label{es}
     \displaystyle\frac{d^{\alpha+1} M_{L_{n}}}{dt}(t)=\displaystyle\sum_{i=1}^{\alpha+1}C_{\alpha+1,i}\hspace{0.05cm}p^{i}e^{it}\displaystyle\frac{d^{i} M_{L_{n}}}{du}(u(t)),
\end{equation}
 
\noindent 
with $C_{\alpha+1,\alpha+1}=C_{\alpha+1,1}=1$ and $C_{\alpha+1,i}=C_{\alpha,i-1}+iC_{\alpha,i}$ for $i\in \{2,3, \dots ,\alpha\}$, which completes the proof.
\end{proof}

\noindent
A special case of Proposition \ref{t1} is the following result, which is valid when $t=0$ in (\ref{s}).
\begin{corollary} \label{c3}
For $\alpha \in \{1,2,\dots\}$, it is verified that \\
\begin{center}
    $\displaystyle\frac{d^{\alpha} M_{L_{n}}}{dt}(0)=\displaystyle\sum_{i=1}^{\alpha}C_{\alpha,i}\hspace{0.05cm}p^{i}\displaystyle\frac{d^{i} M_{L_{n}}}{du}(1)$,
\end{center}
with $C_{\alpha,\alpha}=C_{\alpha,1}=1$ and $C_{\alpha,i}=C_{\alpha-1,i-1}+iC_{\alpha-1,i}$ for $i\in \{2,3, \dots ,\alpha -1\}$.
\end{corollary}

 \begin{theorem}\label{h}
For $n\geq 1$, $p \in (0,1)$ and  $\alpha \in \{1,2,\dots\}$ the following identity holds\\
\begin{center}
    $\displaystyle\frac{d^{\alpha} M_{L_{n}}}{dt}(0)= p^{\alpha}n^{\alpha}+\frac{\alpha}{2}(\alpha(1-p) -p+5)p^{\alpha-1}n^{\alpha-1}+O(n^{\alpha-2})$.
\end{center}
\end{theorem}

\begin{proof}
First, observe that \\
\begin{center}
    $M_{L_{n}}(u(t))=\frac{1}{p^{3}}(u(t)^{n+2}+3(p-1)u(t)^{n+1}+3(p-1)^{2}u(t)^{n}+(p-1)^{3}u(t)^{n-1}),$
\end{center}
\medskip
\noindent
$t \in \mathbb{R}$. Then for each $i \in \{1,2,\dots\}$, we obtain

\begin{center}
    $\displaystyle\frac{d^{i} M_{L_{n}}}{du}(u(0))=\displaystyle\sum _ {k=-1}^{2}\frac{b_{k}}{p^{3}}(n+k)(n+k-1) \ldots (n+k - (i-1))u(0)^{n+k-i}$,
\end{center}
\medskip
\noindent
with $b_{-1}=(p-1)^{3}$, $b_{0}=3(p-1)^{2}$, $b_{1}=3(p-1)$ and $b_{2}=1$.

\medskip
\noindent
Accordingly, it follows that

\hfill

\begin{center}
$\displaystyle\frac{d^{i} M_{L_{n}}}{du}(1)=\frac{1}{p^{3}}\displaystyle\sum _ {k=-1}^{2}b_{k}n^{i}+\frac{1}{2p^{3}}\displaystyle\sum _ {k=-1}^{2}b_{k}i(2k+1-i)n^{i-1}+ O(n^{i-2})$
\end{center}

$~~~~~~~~~~~~~~~~~~=n^{i}-\frac{i(ip+p-6)}{2p}n^{i-1}+ O(n^{i-2}). $

\medskip
\noindent
By Corollary \ref{c3}, it is concluded that for each $\alpha \in \{1,2,\dots\}$

\hfill

\begin{center}
    $\displaystyle\frac{d^{\alpha} M_{L_{n}}}{dt}(0)= p^{\alpha}n^{\alpha}-\frac{\alpha}{2p}(\alpha p+p-6)p^{\alpha}n^{\alpha-1}+C_{\alpha,\alpha -1}\hspace{0.05cm}p^{\alpha-1}n^{\alpha-1}+O(n^{\alpha-2})$,
\end{center}

\medskip
\noindent
as $C_{\alpha,\alpha-1}=C_{\alpha-1,\alpha-2}+(\alpha-1)C_{\alpha-1,\alpha-1}=C_{\alpha-1,\alpha-2}+\alpha-1=\frac{\alpha(\alpha-1)}{2}$. Then

\hfill

\begin{center}
    $\displaystyle\frac{d^{\alpha} M_{L_{n}}}{dt}(0)= p^{\alpha}n^{\alpha}+\frac{\alpha}{2}(\alpha(1-p) -p+5)p^{\alpha-1}n^{\alpha-1}+O(n^{\alpha-2})$.
\end{center}
\end{proof}

\medskip
\noindent
In consequence, by Theorem \ref{h}, the first two moments of $Z_{n}^{g}$ for $\alpha \in \{3, 4,\dots\}$ are given by:
 
\hfill
 
 \begin{center}
     $\mathbb{E}(Z^{g}_{n})=  p^{\alpha}n^{\alpha}+\frac{\alpha}{2}(\alpha(1-p) -p+5)p^{\alpha-1}n^{\alpha-1}+O(n^{\alpha-2})$ and
 \end{center}
 
 \hfill
  
 \begin{center}
     $\mathbb{E}((Z^{g}_{n})^{2})=  p^{2\alpha}n^{2\alpha}+\alpha(2\alpha(1-p) -p+5)p^{2\alpha-1}n^{2\alpha-1}+O(n^{2\alpha-2})$.
 \end{center}
 
\medskip
\noindent
Then,

\hfill

 \begin{center}
     $\mathbb{V}(Z^{g}_{n})=  \alpha^{2}(1-p) p^{2\alpha-1}n^{2\alpha-1}+O(n^{2\alpha-2})$.
 \end{center}

\begin{theorem}\label{s7}
 For any $\alpha \in \{3,4, \dots\}$, it is verified that $\frac{Z^{g}_{n}}{n^\alpha}\xrightarrow{P}p^{\alpha}$, when $n$ goes to infinity.
\end{theorem}

\begin{proof}
Let $X_{n}=\frac{1}{n^\alpha}(Z^{g}_{n}-\frac{\alpha}{2}(\alpha(1-p) -p+5)p^{\alpha-1}n^{\alpha-1}-O(n^{\alpha-2}))$ by Chebyshev’s inequality \cite{17} we get

\begin{center}
    $\mathbb{P}(|X_{n}-p^{\alpha}|\geq \epsilon) \leq \frac{1}{n^{2\alpha}\epsilon^{2}}(\alpha^{2}(1-p) p^{2\alpha-1}n^{2\alpha-1}+O(n^{2\alpha-2}))$,
\end{center}
for any $\epsilon >0$. If $n \rightarrow \infty$ then $\frac{1}{n^{2\alpha}\epsilon^{2}}(\alpha^{2}(1-p) p^{2\alpha-1}n^{2\alpha-1}+O(n^{2\alpha-2})) \rightarrow 0 $, so $X_{n} \xrightarrow{P} p^{\alpha}$. On the other hand, $X_{n}=\frac{Z^{g}_{n}}{n^\alpha}-x_{n}$ with  $x_{n}= \frac{1}{n^\alpha}(\frac{\alpha}{2}(\alpha(1-p) -p+5)p^{\alpha-1}n^{\alpha-1}+O(n^{\alpha-2}))$ then $x_{n} \rightarrow 0$ as $n\rightarrow \infty$. Therefore, the result follows.
\end{proof}
\begin{corollary}\label{f}
For any $\alpha \in \{3,4, \dots\}$ and $r > 0$, it is verified that $\frac{Z^{g}_{n}}{n^\alpha}\xrightarrow{L_{r}}p^{\alpha}$, when $n$ goes to infinity.
\end{corollary}
\begin{proof}
 Using Theorem \ref{s7} we have $\frac{Z^{g}_{n}}{n^\alpha}\xrightarrow{P}p^{\alpha}$ when $n$ goes to infinity. On the other hand, $\frac{Z^{g}_{n}}{n^\alpha} \geq 0$ for all $n\geq 1$, then $|\frac{Z^{g}_{n}}{n^\alpha}|=\frac{Z^{g}_{n}}{n^\alpha}$. For each $r > 0$ there exists $N \in \{1,2, \dots \}$ such that $N > r$. By Theorem 4.2 in Chapter 5 of \cite{17} we obtain $\{\left(\frac{Z^{g}_{n}}{n^\alpha}\right)^{r},~ n \geq 1\}$ is uniformly integrable. Then, applying Theorem 5.4 in Chapter 5 of \cite{17} with $X_{n}=\frac{Z^{g}_{n}}{n^\alpha}$, we obtain the convergence in $r-$mean and the proof is completed. 
\end{proof}

\subsubsection{ Zagreb index}\label{sub3.1}

At time $n \geq 1$, taking $h(x)=x$ and $\alpha = 2$ in (\ref{sa}) we obtain the Zagreb index ($Z_{n}$). According (\ref{j}) we have 

\begin{equation}\label{e5}
     Z_{n}= L_{n}^{2}-3L_{n}+4(n+2).
\end{equation}
\medskip
\noindent  
We can obtain the moments of $Z_{n}$ by (\ref{e5}). Clearly, for $n\geq 1$
\medskip

\begin{center}
    $\mathbb{E}(Z_{n})= n^{2}p^{2}+(-3p^{2}+4p+4)n+2p^{2}-4p+8$ and

\hfill

$\mathbb{V}(Z_{n})= (-4p^{4}+4p^{3})n^{3}+(22p^{4}-40p^{3}+18p^{2})n^{2}+(-38p^{4}+92p^{3}-70p^{2}+16p)n+20p^{4}-56p^{3}+52p^{2}-16p$.
\end{center}

\begin{proposition}\label{kk}
For $p \in (0,1)$, when $n$ goes to infinity it is verified that
\begin{enumerate}
    \item For all $k \in \mathbb{R}$, $\displaystyle\frac{Z_{n}-n^{2}p^{2}}{2\sqrt{p^{3}(1-p)(n+k)^{3}}} \xrightarrow{D} N(0,1)$.
    \item For all $r >0$, $\frac{Z_{n}}{n^2}\xrightarrow{L_{r}}p^{2}$.
\end{enumerate}
\end{proposition}

\begin{proof}
a) By (\ref{e5}) $Z_{n}=\left(L_{n}-\frac{3}{2}\right)^{2}+\frac{16n+23}{4}$ for each $n \geq 1$. Note that Proposition \ref{p1} c implies that $\displaystyle\frac{L_{n}-\frac{3}{2}}{\sqrt{p(1-p)(n+k)}}$ is equivalent to a normal random variable with mean $\displaystyle\frac{\frac{3}{2}+(n-1)p}{\sqrt{p(1-p)(n+k)}}$ and variance $1$ in distribution. Therefore, $\left( \displaystyle\frac{L_{n}-\frac{3}{2}}{\sqrt{p(1-p)(n+k)}}\right)^{2} \sim \chi^{2}
\left(\displaystyle\frac{\left(\frac{3}{2}+(n-1)p\right)^{2}}{p(1-p)(n+k)},1\right)$. Indicating us that:
\begin{center}
    $\displaystyle\frac{Z_{n}}{p(1-p)(n+k)}-\displaystyle\frac{16n+23}{4p(1-p)(n+k)} \sim \chi^{2}
\left(\displaystyle\frac{\left(\frac{3}{2}+(n-1)p\right)^{2}}{p(1-p)(n+k)},1\right)$.
\end{center}
\medskip
\noindent
By the well-known normal approximation
of non-central chi-squared distribution (see \cite{11}), it is obtained that
\begin{center}
    $\displaystyle\frac{\displaystyle\frac{Z_{n}}{p(1-p)(n+k)}-\displaystyle\frac{16n+23}{4p(1-p)(n+k)}-(1+\displaystyle\frac{\left(\frac{3}{2}+(n-1)p\right)^{2}}{p(1-p)(n+k)})}{\sqrt{2(1+\displaystyle\frac{2\left(\frac{3}{2}+(n-1)p\right)^{2}}{p(1-p)(n+k)})}} \xrightarrow{D} N(0,1)$,
\end{center}

\medskip
\noindent
as $n \rightarrow \infty$. In particular,
\begin{center}
$\displaystyle\frac{\displaystyle\frac{Z_{n}}{p(1-p)(n+k)}-\displaystyle\frac{16n+23}{4p(1-p)(n+k)}-1-\displaystyle\frac{\left(\frac{3}{2}+(n-1)p\right)^{2}}{p(1-p)(n+k)}}{\sqrt{2(1+\displaystyle\frac{2\left(\frac{3}{2}+(n-1)p\right)^{2}}{p(1-p)(n+k)})}}$
\end{center}
\begin{center}
$~~~~~~~~~~~=\displaystyle\frac{Z_{n}-\frac{16n+23}{4}-p(1-p)(n+k)-\left(\frac{3}{2}+(n-1)p\right)^{2}}{\sqrt{p(1-p)(n+k)}\sqrt{2(p(1-p)(n+k)+2\left(\frac{3}{2}+(n-1)p\right)^{2})}}$
\end{center}

\begin{center}
$=\displaystyle\frac{Z_{n}-\frac{16n+23}{4}-p(1-p)(n+k)-\left(\frac{3}{2}+(n-1)p\right)^{2}}{\sqrt{4p^{3}(1-p)(n+k)^{3}+o(n^{3})}}$
\end{center}
~~~~~~~~~~~~~~~~~~~~$=a_{n}\displaystyle\frac{Z_{n}-n^{2}p^{2}}{\sqrt{4p^{3}(1-p)(n+k)^{3}}}+b_{n}$,

\medskip
\noindent
where
\begin{center}
$a_{n}=\displaystyle\frac{\sqrt{4p^{3}(1-p)(n+k)^{3}}}{\sqrt{4p^{3}(1-p)(n+k)^{3}+o(n^{3})}}$,
\end{center}

\begin{center}
    $b_{n}=\displaystyle\frac{o(n^{\frac{3}{2}})}{\sqrt{4p^{3}(1-p)(n+k)^{3}+o(n^{3})}}$.
\end{center}
\medskip
\noindent
It is verified that $a_{n} \rightarrow 1$ and $b_{n} \rightarrow 0$ as $n \rightarrow \infty$, then

\begin{center}
$\displaystyle\frac{Z_{n}-n^{2}p^{2}}{2\sqrt{p^{3}(1-p)(n+k)^{3}}} \xrightarrow{D} N(0,1)$.
\end{center}
b) The proof can be done similarly to that of Corollary \ref{f}.
\end{proof}

\paragraph{Gordon-Scantlebury Index}

Defining $S_{n}$ as the Gordon-Scantlebury index at time $n \geq 1$, which verifies that $Z_{n}=2(S_{n}+E_{n})$  \cite{10} where $E_{n}$ is the number of edges at time $n$. The tree generated by the model at time $n$ has $n+3$ nodes then at time $n$ it has $n+2$ edges, thus $S_{n}=\displaystyle\frac{Z_{n}}{2}-n-2$. For $n\geq 1$, we get
\begin{center}
     $\mathbb{E}(S_{n})= \frac{n^{2}p^{2}}{2}+(-\frac{3}{2}p^{2}+2p+1)n+p^{2}-2p+2$ and
          
\end{center}

\hfill

\begin{center}
     $\mathbb{V}(S_{n})= (-p^{4}+p^{3})n^{3}+(\frac{11}{2}p^{4}-10p^{3}+\frac{9}{2}p^{2})n^{2}+(-\frac{19}{2}p^{4}+23p^{3}-\frac{35}{2}p^{2}+4p)n+5p^{4}-14p^{3}+13p^{2}-4p$.
     
\end{center}

\begin{proposition}\label{p8}
For $p \in (0,1)$, when $n$ goes to infinity it is verified that
\begin{enumerate}
    \item For all $k \in \mathbb{R}$, $\displaystyle\frac{S_{n}-\frac{n^{2}p^{2}}{2}}{\sqrt{p^{3}(1-p)(n+k)^{3}}} \xrightarrow{D} N(0,1)$.
    \item For all $r>0$, $\frac{S_{n}}{n^{2}}\xrightarrow{L_{r}}\frac{p^{2}}{2}.$
\end{enumerate}
\end{proposition}

\paragraph{Platt Index}
Let $P_{n}$ denote the Platt index at time $n \geq 1$, wich verifies that $P_{n}=2S_{n}$ \cite{10}. Thus we obtain that\\

\begin{center}
     $\mathbb{E}(P_{n})= n^{2}p^{2}+(-3p^{2}+4p+2)n+2p^{2}-4p+4$ and
\end{center}

\hfill

\begin{center}
     $\mathbb{V}(P_{n})=\mathbb{V}(Z_{n}) $.
     
\end{center}
\begin{proposition}
For $p \in (0,1)$, when $n$ goes to infinity it is verified that
\begin{enumerate}
    \item For all $k \in \mathbb{R}$, $\displaystyle\frac{P_{n}-n^{2}p^{2}}{2\sqrt{p^{3}(1-p)(n+k)^{3}}} \xrightarrow{D} N(0,1)$.
    \item For all $r >0$,  $\frac{P_{n}}{n^{2}} \xrightarrow{L_{r}} p^{2}.$
\end{enumerate}

\end{proposition}

\subsubsection{Forgotten Index}

At time $n \geq 1$, taking $h(x)=x$ and $\alpha = 3$ in (\ref{sa}) we obtain the Forgotten Index ($F_{n}$). According (\ref{j}) we have \\

\begin{center}
     $F_{n}= L_{n}^{3}-7L_{n}+8(n+2)$.
\end{center}

\medskip
\noindent
Our next task is to calculate the first moment of $F_{n}$, and consequently to get the variance of $F_{n}$, 

\hfill
\begin{center}
     $\mathbb{E}(F_{n})= n^{3}p^{3}+(12p^{2}-6p^{3})n^{2}+(11p^{3}-36p^{2}+30p+8)n-6p^{3}+24p^{2}-30p+22$
\end{center}

\medskip
\noindent
and

\hfill

\begin{center}
     $\mathbb{V}(F_{n})=9p^{5}(1-p)n^{5}-9p^{4}(1-p)(13p-18)n^{4}+3p^{3}(1-p)(197p^{2}-490p+302)n^{3}-9p^{2}(1-p)(159p^{3}-530p^{2}+572p-192)n^{2}-6p(1-p)^{2}(272p^{3}-803p^{2}+714p-150)n+36p(1-p)^{2}(19p^3-64p^2+71p-25)$.
     
\end{center}
\medskip
\noindent
According to Corollary \ref{f}, we obtain the following result.
\begin{corollary}
For all $r >0$ it is verified that, $\frac{F_{n}}{n^{3}}\xrightarrow{L_{r}}p^{3}$  as $n \rightarrow \infty$.
\end{corollary}

\subsection{Degree-based Gini Index} \label{hu}
Recently, a degree-based Gini index for general graphs was proposed by \cite{12}.  This index is a topological measure on a graph capturing the
proximity to regular graphs. In \cite{2} the authors considered the degree-based Gini index
introduced by \cite{12}, with slight modifications. In this section, we study the degree-based Gini index defined in \cite{2}. By definition, the degree-based Gini index of a graph  within the class of RSTs  at time $n \geq 1$ is given by \\
\begin{center}
    $G_{n}=\displaystyle\frac{\displaystyle\sum_{i,j \in V_{n}}|deg_{i}-deg_{i}|}{(n+3)^{2}\mathbb{E}(deg_{v^{*}})},$
\end{center}

\medskip
\noindent
where $v^{*}$ is an arbitrary node of a randomly selected graph
from class of RSTs and $V_{n}$ denotes the node set at time $n$. We take $\mathbb{E}(G_{n})$ as the degree-based Gini index of the class. Due to the characteristics of the model, we get

\hfill

\begin{center}
    $~~\displaystyle\sum_{i,j \in V_{n}}|deg_{i}-deg_{i}|=|L_{n}-1|L_{n}+|L_{n}-2|(n+2-L_{n})+L_{n}(n+2-L_{n})$
\end{center}
$~~~~~~~~~~~~~~~~~~~~~~~~~~~=-L_{n}^{2}+(2n+5)L_{n}-2n-4,$

\medskip
\noindent
since $L_{n} \geq 3$ for all $n \geq 1$. Finally,

\hfill

\begin{center}
    $\mathbb{E}(d_{v^{*}})=\frac{2(n+2)}{n+3}$,
\end{center}
\medskip
\noindent
thus,

\begin{equation}\label{ilo}
     G_{n}=\frac{-L_{n}^{2}+(2n+5)L_{n}-2n-4}{2(n+3)(n+2)}.
\end{equation}

\medskip
\noindent
It follows that

\hfill

\begin{center}
     $\mathbb{E}(G_{n})=\frac{(2p-p^{2})n^{2}+(3p^{2}-4p+4)n-2p^{2}+2p+2}{2(n+3)(n+2)}$.
\end{center}
\medskip
\noindent
Next, we get an asymptotic property of the degree-based Gini index of the class of RSTs at time $n$.
\hfill

\begin{proposition}\label{p20}
As $n \rightarrow \infty$, we have $\mathbb{E}(G_{n}) \rightarrow \frac{p(2-p)}{2}$.
\end{proposition}

\medskip
\noindent
We see from (\ref{ilo}) that

\begin{center}
    $G_{n}=\frac{-\left(L_{n}-\frac{2n+5}{2}\right)^{2}+\left(\frac{2n+5}{2}\right)^{2}-2n-4}{2(n+3)(n+2)}.$
\end{center}

\medskip
\noindent
Therefore,
\medskip
\begin{center}
    $\mathbb{V}(G_{n})=\frac{4p(1-p)^{3}n^{3}+2p(11p-6)(1-p)^{2}n^{2}+2p(1-p)(19p^{2}-23p+6)n-4p(1-p)(5p^{2}-5p+1)}{4(n+3)^{2}(n+2)^{2}}.$
\end{center}

\begin{theorem}\label{y}
It is verified that $G_{n}\xrightarrow{P} \frac{p(2-p)}{2}$, as $n \rightarrow \infty$.
\end{theorem}
\begin{proof}
By Chebyshev's inequality \cite{17} we have 

\begin{center}
    $\mathbb{P}(|G_{n}-\mathbb{E}(G_{n})|\geq \epsilon) \leq \frac{\mathbb{V}(G_{n})}{\epsilon^{2}}$,
\end{center}
for any $\epsilon > 0$. If $n \rightarrow \infty$ then $\frac{\mathbb{V}(G_{n})}{\epsilon^{2}} \rightarrow 0$, so $G_{n}-\mathbb{E}(G_{n}) \xrightarrow{P} 0$. Therefore, Proposition \ref{p20} completes the proof. 
\end{proof}
\begin{corollary}\label{IQ}
For all $r >0$, we have $G_{n}\xrightarrow{L_{r}} \frac{p(2-p)}{2}$, when $n$ goes to infinity.
\end{corollary}
\begin{proof}
Argued in a similar manner to Corollary \ref{f} by Theorem \ref{y}, the result follows.
\end{proof}

\noindent

\subsection{Degree-based Hoover index} 
In \cite{15} the authors proposed a degree-based Hoover index for graphs analogous to the degree-based Gini index introduced in \cite{12} as a competing measure for assessing graph regularity. In our context, at time $n \geq 1$ the degree-based Hoover index of a graph  within the class of RSTs ($H_{n}$) is defined as follows:\\
\begin{center}
    $H_{n}=\displaystyle\frac{\displaystyle\sum_{i \in V_{n}}|(n+3)deg_{i}-2
(n+2)|}{4(n+2)(n+3)},$
\end{center}

\medskip
\noindent
where $V_{n}$ denotes the node set at time $n$. In a similar way, we take $\mathbb{E}(H_{n})$ as the degree-based Hoover index of the class. The same analysis applied in Section \ref{hu} is used in this section and we obtain the following results.

\begin{proposition}\label{T}
 For $n \geq 1$, we have
\begin{enumerate}

    \item  $H_{n}=\frac{(n+1)L_{n}}{2(n+3)(n+2)}.$
    \item $\mathbb{E}(H_{n})=\frac{
     pn^{2}+3n+3-p}{2(n+3)(n+2)}$.
     \item $\mathbb{V}(H_{n})=\frac{p(1-p)(n^{2}-1)(n+1)}{4(n+3)^{2}(n+2)^{2}}.$
     \item For all $r >0$, $H_{n}\xrightarrow{L_{r}} \frac{p}{2}$, as $n \rightarrow \infty$.
\end{enumerate}
\end{proposition}

\noindent
A direct consequence of Proposition \ref{T} is the following corollary.
\begin{corollary}
In the preferential model for all $r >0$, it is verified that $H_{n} \xrightarrow{L_{r}}  \frac{1}{4}$ and $\mathbb{E}(H_{n}) \rightarrow \frac{1}{4}$ when $n$ goes to infinity.
\end{corollary}

\section{Conclusion}\label{sec13}
We investigate a class of RSTs, the random
variable of prime interest is the number of leaves as time proceeds and we calculate the moment generating function of the leaves and show that the number of leaves follow a Gaussian  law asymptotically. Next, we investigate several useful topological  indices for this
class, including degree-based Gini index, degree-based Hoover index, generalized Zagreb index and other indices associated with these. Moreover, Proposition 3 and Theorem 1 showed in \cite{2} are deduced from Proposition \ref{p1} taking $p=1/2$.  In similar way, the results exposed in Section 3.2.1 and 3.2.2 of \cite{2} are obtained as a special case of the results demonstrated in Section \ref{pa} and \ref{hu}, respectively. 

\bibliographystyle{chicago}
\bibliography{Bib}

\begin{thebibliography}{}

\bibitem[\protect\citeauthoryear{Aguilar-S{\'{a}}nchez,
  M{\'{e}}ndez-Berm{\'{u}}dez, Rodr{\'{\i}}guez, and
  Sigarreta}{Aguilar-S{\'{a}}nchez et~al.}{2021}]{24}
Aguilar-S{\'{a}}nchez, R., J.~A. M{\'{e}}ndez-Berm{\'{u}}dez, J.~M.
  Rodr{\'{\i}}guez, and J.~M. Sigarreta (2021, July).
\newblock Normalized {S}ombor {I}ndices as {C}omplexity {M}easures of {R}andom
  {N}etworks.
\newblock {\em Entropy\/}~{\em 23\/}(8), 976.

\bibitem[\protect\citeauthoryear{Bahls, Lake, and Wertheim}{Bahls
  et~al.}{2010}]{E1}
Bahls, P., S.~Lake, and A.~Wertheim (2010, October).
\newblock Gracefulness of families of spiders.
\newblock {\em Involve, a Journal of Mathematics\/}~{\em 3\/}(3), 241--247.

\bibitem[\protect\citeauthoryear{Barabási and Albert}{Barabási and
  Albert}{1999}]{18}
Barabási, A.-L. and R.~Albert (1999, oct).
\newblock Emergence of {S}caling in {R}andom {N}etworks.
\newblock {\em Science\/}~{\em 286\/}(5439), 509--512.

\bibitem[\protect\citeauthoryear{Domicolo and Mahmoud}{Domicolo and
  Mahmoud}{2020}]{12}
Domicolo, C. and H.~Mahmoud (2020).
\newblock Degree-based {G}ini index for graphs.
\newblock {\em Probability in the Engineering and Informational
  Sciences\/}~{\em 34}, 157--171.

\bibitem[\protect\citeauthoryear{Ducoffe, Marinescu-Ghemeci, Obreja, Popa, and
  Tache}{Ducoffe et~al.}{2018}]{r7}
Ducoffe, G., R.~Marinescu-Ghemeci, C.~Obreja, A.~Popa, and R.~M. Tache (2018).
\newblock Extremal {G}raphs with respect to the {M}odified {F}irst {Z}agreb
  {C}onnection {I}ndex.
\newblock {\em 20th International Symposium on Symbolic and Numeric Algorithms
  for Scientific Computing (SYNASC)\/}, 141--148.

\bibitem[\protect\citeauthoryear{Frati, Geyer, and Kaufmann}{Frati
  et~al.}{2009}]{r3}
Frati, F., M.~Geyer, and M.~Kaufmann (2009).
\newblock Planar packing of trees and spider trees.
\newblock {\em Information processing letters\/}~{\em 109\/}(6), 301--307.

\bibitem[\protect\citeauthoryear{Garc{\'\i}a, Hernando, Hurtado, Noy, and
  Tejel}{Garc{\'\i}a et~al.}{2002}]{r1}
Garc{\'\i}a, A., C.~Hernando, F.~Hurtado, M.~Noy, and J.~Tejel (2002).
\newblock Packing trees into planar graphs.
\newblock {\em Journal of Graph Theory\/}~{\em 40\/}(3), 172--181.

\bibitem[\protect\citeauthoryear{Geyer, Hoffmann, Kaufmann, Kusters, and
  Tóth}{Geyer et~al.}{2017}]{r4}
Geyer, M., M.~Hoffmann, M.~Kaufmann, V.~Kusters, and C.~D. Tóth (2017).
\newblock The planar tree packing theorem.
\newblock {\em Journal of Computational Geometry\/}, Vol. 8 No. 2 (2017):
  Special Issue of Selected Papers from SoCG 2016.

\bibitem[\protect\citeauthoryear{Gut}{Gut}{2005}]{17}
Gut, A. (2005).
\newblock {\em Probability: {A} {G}raduate {C}ourse}.
\newblock New York: Springer.

\bibitem[\protect\citeauthoryear{Gutman and Das}{Gutman and Das}{2004}]{10}
Gutman, I. and K.~C. Das (2004).
\newblock The {f}irst {Z}agreb {i}ndex 30 {y}ears {a}fter.
\newblock {\em MATCH Commun. Math. Comput. Chem\/}~{\em 50\/}(1), 83--92.

\bibitem[\protect\citeauthoryear{Jampachon, Nakprasit, and
  Poomsa-ard}{Jampachon et~al.}{2014}]{S1}
Jampachon, P., K.~Nakprasit, and T.~Poomsa-ard (2014, 01).
\newblock Graceful labeling of some classes of spider graphs with three legs
  greater than one.
\newblock {\em Thai J. Math\/}~{\em 12}, 621--630.

\bibitem[\protect\citeauthoryear{Kazemi}{Kazemi}{2021}]{38}
Kazemi, R. (2021, March).
\newblock Gordon-{S}cantlebury and {P}latt {I}ndices of {R}andom
  {P}lane-oriented {R}ecursive {T}rees.
\newblock {\em Mathematics Interdisciplinary Research\/}~{\em 6\/}(1).

\bibitem[\protect\citeauthoryear{Li, Shi, and Gao}{Li et~al.}{2020}]{36}
Li, S., L.~Shi, and W.~Gao (2020, December).
\newblock Topological indices computing on random chain structures.
\newblock {\em International Journal of Quantum Chemistry\/}~{\em 121\/}(8).

\bibitem[\protect\citeauthoryear{Li, Shi, and Gao}{Li et~al.}{2021}]{37}
Li, S., L.~Shi, and W.~Gao (2021, January).
\newblock Two modified {Z}agreb indices for random structures.
\newblock {\em Main Group Metal Chemistry\/}~{\em 44\/}(1), 150--156.

\bibitem[\protect\citeauthoryear{Mart{\'{\i}}nez-Mart{\'{\i}}nez,
  M{\'{e}}ndez-Berm{\'{u}}dez, Rodr{\'{\i}}guez, and
  Sigarreta}{Mart{\'{\i}}nez-Mart{\'{\i}}nez et~al.}{2020}]{23}
Mart{\'{\i}}nez-Mart{\'{\i}}nez, C., J.~M{\'{e}}ndez-Berm{\'{u}}dez, J.~M.
  Rodr{\'{\i}}guez, and J.~M. Sigarreta (2020, July).
\newblock Computational and analytical studies of the {R}andi{\'{c}} index in
  {E}rd\"{o}s{\textendash}{R}{\'{e}}nyi models.
\newblock {\em Applied Mathematics and Computation\/}~{\em 377}, 125137.

\bibitem[\protect\citeauthoryear{Oda and Ota}{Oda and Ota}{2006}]{r2}
Oda, Y. and K.~Ota (2006).
\newblock Tight planar packings of two trees.
\newblock In {\em European Workshop on Computational Geometry}, pp.\  215--216.

\bibitem[\protect\citeauthoryear{Pegu, Deka, Gogoi, and Bharali}{Pegu
  et~al.}{2021}]{PP1}
Pegu, A., B.~Deka, I.~J. Gogoi, and A.~Bharali (2021).
\newblock Two generalized topological indices of some graph structures.
\newblock {\em J. Math. Comput. Sci.\/}~{\em 11\/}(5), 5549--5564.

\bibitem[\protect\citeauthoryear{Ren, Zhang, and Dey}{Ren et~al.}{2021}]{2}
Ren, Y., P.~Zhang, and D.~K. Dey (2021).
\newblock Investigating {S}everal {F}undamental {P}roperties of {R}andom
  {L}obster {T}rees and {R}andom {S}pider {T}rees.
\newblock {\em Methodol Comput Appl Probab.\/}.

\bibitem[\protect\citeauthoryear{Severo and Zelen}{Severo and Zelen}{1960}]{11}
Severo, N.~C. and M.~Zelen (1960).
\newblock Normal approximation to the chi-square and non-central {F}
  probability functions.
\newblock {\em Biometrika\/}~{\em 47\/}(3/4), 411--416.

\bibitem[\protect\citeauthoryear{Shiu}{Shiu}{2008}]{0}
Shiu, W.~C. (2008).
\newblock Extremal hosoya index and merrifield--simmons index of hexagonal
  spiders.
\newblock {\em Discrete applied mathematics\/}~{\em 156\/}(15), 2978--2985.

\bibitem[\protect\citeauthoryear{Zhang and Wang}{Zhang and Wang}{2021}]{15}
Zhang, P. and X.~Wang (2021).
\newblock Several {T}opological {I}ndices of {R}andom {C}aterpillars.
\newblock {\em Methodol Comput Appl Probab.\/}.

\end{thebibliography}

\end{document}